\documentclass[12pt]{amsart}
\usepackage{amssymb, latexsym}
\usepackage{enumerate}
\usepackage{amsmath}
\usepackage{dsfont}
\usepackage{a4wide}
\usepackage{mathrsfs}
\usepackage[cp1250]{inputenc}
\usepackage[T1]{fontenc}
\usepackage{tabularx}
\usepackage{multirow}
\usepackage{scalerel}
\usepackage{etoolbox}
\patchcmd{\section}{\normalfont}{\normalfont\Large}{}{}
\patchcmd{\section}{\scshape}{\bfseries}{}{}
\makeatletter
\renewcommand{\@secnumfont}{\bfseries}
\makeatother
\usepackage{lmodern}
\usepackage{hyperref}
\hypersetup{hypertexnames=false}
\let\originalforall=\forall
\renewcommand{\forall}{\mathop{\vcenter{\hbox{\Large$\originalforall$}}}}

\let\originalexists=\exists
\renewcommand{\exists}{\mathop{\vcenter{\hbox{\Large$\originalexists$}}}}

\newtheorem{thm}{Theorem}[]
\newtheorem{cor}[thm]{Corollary}
\newtheorem{lem}[thm]{Lemma}
\newtheorem{prop}[thm]{Proposition}

\newtheorem{prob}{Problem}
\newtheorem{de}[thm]{Definition}
\newtheorem{rem}[thm]{Remark}

\date{}
\title{An invitation to cortex algebras and spectral permanence property}
\author{Przemys{\l}aw Ohrysko}
\address{University of Warsaw, Faculty of Mathematics, Banacha 2A, 02-097 Warsaw, Poland}
\email{p.ohrysko@gmail.com}
\begin{document}
\begin{abstract}
We initiate the study of cortex algebras (commutative Banach algebras with extendable multiplicative-linear functionals) and the spectral permanence property. In addition, we analyze some particular examples in this context and present open problems.
\end{abstract}

\maketitle
\section{Introduction}
We start with collecting some basic notation. For the proofs consult textbooks \cite{k},\cite{r2} and \cite{zel}.
\\
Throughout the paper, the letter $A$ will always denote a commutative unital complex Banach algebra with the unit $e$. The Gelfand space of $A$ (the set of all non-zero multiplicative-linear functionals on $A$) will be abbreviated as $\triangle(A)$. Equipped with the weak$^{\ast}$-topology it is a compact Hausdorff space. There is a one-to-one correspondence between $\varphi\in\triangle(A)$ and maximal ideals in $A$ given by $\varphi\to\mathrm{ker}\varphi$ and hence $\triangle(A)$ is often called a maximal ideal space of $A$.
\\
The spectrum of $x\in A$ is defined as the set:
\begin{equation*}
    \sigma(x)=\{\lambda\in\mathbb{C}:\lambda e-x\text{ is not invertible in }A\}.
\end{equation*}
If there is more than one Banach algebra involved, for example, $A$ is a closed subalgebra of $B$, then $\sigma_{A}(x)$ and $\sigma_{B}(x)$ will refer to spectra calculated in $A$ or $B$ (respectively). Recall that the spectrum of any $x\in A$ is a non-empty compact subset of the complex plane. By the Gelfand theory
\begin{equation*}
    \sigma(x)=\widehat{x}\left(\triangle(A)\right),
\end{equation*}
where for every $x\in A$ we define the Gelfand transform $\widehat{x}:\triangle(A)\to\mathbb{C}$ of $x$ by the formula:
\begin{equation*}
    \widehat{x}(\varphi):=\varphi(x)\text{ for }\varphi\in\triangle(A).
\end{equation*}
It is well-known that the spectrum of an element $x\in A$ may change if we calculate it in a larger algebra $B$ containing $A$ as a closed subalgebra. In this regard, we introduce the following definition.
\begin{de}\label{spm}
    Let $A$ be a unital commutative Banach algebra. We say that $A$ has a \textbf{spectral permanence property}, if for every commutative Banach algebra $B$ containing $A$ as a closed subalgebra and for every $x\in A$ we have:
    \begin{equation*}
        \sigma_{A}(x)=\sigma_{B}(x).
    \end{equation*}
\end{de}
In the literature, one may find a notation of a permanent spectrum of an element $x\in A$ as the intersection of spectras of $x$ computed in all algebras containing $A$ as a closed subalgebra. In these terms, an algebra $A$ has a spectral permanence property, if every element has its spectrum equal to the permanent spectrum.
\\
It turns out that the spectral properties of elements in a Banach algebra are related to the \textbf{cortex} of a Banach algebra introduced by R. Arens in \cite{a} which is defined as the set of all multiplicative-linear functionals, which extend to any superalgebra containing the original algebra as a closed subalgebra and abbreviated $\mathrm{cor}(A)$ for $A$ unital commutative Banach algebra $A$.
It is an easy exercise to check that $\mathrm{cor}(A)$ is a closed subset of $\triangle(A)$. We are ready now to present the second main definition needed later on.
\begin{de}\label{cor}
    A unital commutative Banach algebra $A$ is called a \textbf{cortex algebra}, if $\mathrm{cor}(A)=\triangle(A)$.
\end{de}
In the next section, we will present basic results on spectral permanence property and cortex algebras supplemented with some examples. The third section will be devoted to showing that measure algebras have a spectral permanence property. In the last part, we will briefly discuss Fourier-Stieltjes algebras and other algebras in this context and state some open questions concerning both particular examples and the general relation between Definition \ref{spm} and Definition \ref{cor}.
\section{Basic results and examples}
We start by recalling a standard theorem on the relationship between the spectras of $x\in A$ when $A$ is a closed subalgebra of a Banach algebra $B$ from \cite{r2}.
\begin{thm}
    Let $A$ be a closed subalgebra of a Banach algebra $B$. Then $\sigma_{A}(x)$ is a set-theoretic sum of $\sigma_{B}(x)$ and (possibly empty) family of bounded components of the complement of $\sigma_{B}(x)$.
\end{thm}
The following corollary follows immediately. 
\begin{cor}
    If $\sigma_{A}(x)$ has an empty interior, then $\sigma_{A}(x)=\sigma_{B}(x)$.
    More generally, if for every $x\in A$, $\sigma_{A}(x)$ has an empty interior, then $A$ has a spectral permanence property.
\end{cor}
For purely algebraic reasons, an element $x\in A$ cannot be invertible in any superalgebra if it is a zero divisor. In Banach algebra setting there is a generalization of this notion: an element $x\in A$ is called a \textbf{topological divisor of zero}, if there exists a sequence $x_{n}\in A$ of elements with $\|x_{n}\|=1$ satisfying:
\begin{equation*}
    \lim_{n\to\infty}\|xx_{n}\|=0.
\end{equation*}
The following theorem allows us to formulate an equivalent condition for the spectral permanance property in terms of topological divisors of zero.
\begin{thm}
    Let $A$ be a unital commutative Banach algebra. Then $A$ has a spectral permanence property if and only if every non-invertible element in $A$ is a topological divisor of zero.
\end{thm}
\begin{proof}
    The argument here is standard and goes back to Arens \cite{a}. We present a sketch of the proof for the convenience of the readers. Note that the spectrum of an element $x\in A$ may possibly change when calculated in some larger algebra only if some elements non-invertible in $A$ become invertible in $B$. This, of course, cannot happen if every non-invertible element in $A$ is a topological divisor of zero. In the other direction, we need to prove that if a non-invertible element $x\in A$ is not a topological divisor of zero, then there exists a unital commutative Banach algebra $B$ containing $A$ as a closed subalgebra such that $x$ is invertible in $B$. In order to do so we first form an algebra of absolutely convergent power series over $A$:
    \begin{equation*}
        \widetilde{A}=\{\widetilde{x}(t)=\sum_{n=0}^{\infty}x_{n}t^{n},\text{ }x_{n}\in A,\text{ }\sum_{n=0}^{\infty}\|x_{n}\|<\infty\}.
    \end{equation*}
    Then we consider an ideal $J=(e-xt)\widetilde{A}$ in $\widetilde{A}$ and its closure $I=\overline{J}$. Then $B=\widetilde{A}/I$ is the desired Banach algebra (the details can be found in \cite{a} and \cite{k})
\end{proof}
We now turn our attention to the cortex algebras. We first prove a simple theorem which partially establishes the relation between Definition \ref{spm} and Definition \ref{cor}.
\begin{thm}
    If $A$ is a cortex algebra, then $A$ has a spectral permanence property.
\end{thm}
\begin{proof}
    Let $B$ be a unital, commutative Banach algebra containing $A$ as a closed subalgebra and let $x\in A$. Invertible elements in $A$ are obviously invertible in $B$ giving the inclusion $\sigma_{B}(x)\subset\sigma_{A}(x)$. But if $\lambda\in\sigma_{A}(x)$ then there exists $\varphi_{0}\in\triangle(A)$ such that $\lambda=\varphi_{0}(x)$ and since $A$ is a cortex algebra there exists an extension $\widetilde{\varphi_{0}}\in\triangle(B)$ of $\varphi_{0}$ which implies $\lambda\in\sigma_{B}(x)$ and completes the proof.
\end{proof}
In order to give easily verifiable conditions implying that a Banach algebra is cortex we need to recall the notion of \textbf{Shilov boundary} of $A$: it is the smallest closed subset of $\triangle(A)$ (usually denoted $\partial(A)$) such that
\begin{equation*}
    \|\widehat{x}|_{\partial(A)}\|_{\infty}=\|\widehat{x}\|_{\infty}\text{ for every }x\in A.
\end{equation*}
The existence of the Shilov boundary is a non-trivial fact but nowadays it is usually a part of a basic course devoted to commutative Banach algebras.
There is an interesting connection between the Shilov boundary and the topological divisors of zero (see Corollary 3.4.8 in \cite{k}).
\begin{thm}\label{sd}
    Let $\varphi\in\partial(A)$. Then every $x\in\mathrm{ker}\varphi$ is a topological divisor of zero.
\end{thm}
In our context, the following theorem is of particular interest (see Theorem 3.3.9 in \cite{k}).
\begin{thm}
    Let $A$ be a unital commutative Banach algebra. Then $\partial(A)\subset\mathrm{cor}(A)$.
\end{thm}
Taking into account Definition \ref{cor} we immediately get the following corollary. 
\begin{cor}\label{corc}
    If $\partial(A)=\triangle(A)$, then $A$ is a cortex algebra.
\end{cor}
It should be noted that, in general, we cannot expect the Shilov boundary and cortex to coincide. It was shown by Shilov himself and the detailed discussion of his example, which is a variant of a Beurling algebra, is presented in \cite{a}.
\\
There are many conditions that ensure $\triangle(A)=\partial(A)$, which gives a wealth of examples of cortex algebras. We start with two rich classes of Banach algebras.
\begin{enumerate}
    \item An involutive, commutative Banach algebra $A$ is called \textbf{symmetric}, if $\widehat{x^{\ast}}=\overline{\widehat{x}}$ for every $x\in A$.
    \item A unital, commutative Banach algebra $A$ is called \textbf{regular}, if for every closed subset $C\subset\triangle(A)$ and any $\varphi_{0}\in\triangle(A)\setminus C$ there exists $x\in A$ such that $\widehat{x}|_{C}=0$ and $\widehat{x}(\varphi_{0})\neq 0$.
\end{enumerate}
\begin{prop}\label{propc}
    Let $A$ be unital, commutative Banach algebra. If $A$ is symmetric or regular, then $\partial(A)=\triangle(A)$. In particular, $A$ is a cortex algebra.
\end{prop}
\begin{proof}
    Suppose not, if $A$ is regular, then we find $x\in A$ such that $\widehat{x}|_{\partial(A)}=0$ and $\widehat{x}(\varphi_{0})\neq 0$ for some $\varphi_{0}\in\triangle(A)\setminus\partial(A)$ which contradicts the definition of the Shilov boundary. If $A$ is symmetric, we use Gelfand-Naimark theorem to find $x\in A$ with Gelfand transform satisfying $|\widehat{x}(\varphi)|<\frac{1}{2}$ for $\varphi\in\partial(A)$ and $|\widehat{x}(\varphi_{0})|>\frac{1}{2}$ for some $\varphi_{0}\in\triangle(A)\setminus\partial(A)$ which finishes the proof in the same pattern.
\end{proof}
In view of the last proposition, we are able to list some examples of cortex algebras:
\begin{itemize}
    \item $C(K)$ for a compact space $K$.
    \item The unitization of $L^{1}(G)$ where $G$ is a locally compact Abelian group.
    \item The unitization of $L^{p}(G)$  for $1<p<\infty$ where $G$ is a compact Abelian group.
    \item The unitization of $l^{p}(\mathbb{N})$ for $1\leq p\leq\infty$ with a pointwise multiplication.
\end{itemize}
We skip the discussion of more advanced examples such as Beurling algebras which are regular and thus cortex under some additional conditions (see Section 4.7 in \cite{k}).
\\
Also, it is easy to give some explicit examples of non-cortex algebras. The most classical one is the disc algebra $A(\mathbb{D})$ (the algebra of continuous functions on the closed unit disc that are analytic on the interior of the disc), which can be identified with a closed subalgebra of $C(\mathbb{T})$ (the algebra of continuous functions on the unit circle) via the maximal modulus principle. Clearly $A(\mathbb{D})$ does not have a spectral permanence property as the spectrum of $f(z)=z\in A(\mathbb{D})$ in $A(\mathbb{D})$ is equal to the closed unit disc, but in $C(\mathbb{T})$ the spectrum is equal to the closed unit circle. Other examples include the closure of algebraic polynomials in the annulus in the complex plane. Of course, we can find more complicated non-cortex algebras venturing into the realm of complex analysis in higher dimensions.
\section{Measure algebras}
In this section we will prove that measure algebras have a spectral permanence property. However, since the proof is non-trivial we need first to collect the necessary information. Let $G$ be a locally compact Abelian group with the dual group $\widehat{G}$. The measure algebra $M(G)$ is the set of all complex Borel measures on $G$ equipped with the total variation norm and the convolution product. Measures absolutely continuous with respect to the Haar measure on $G$ can be identified with $L^{1}(G)$ via the Radon-Nikodym theorem. Along with the Gelfand transform of any $\mu\in M(G)$ we consider also its Fourier-Stieltjes transform defined for any $\gamma\in\widehat{G}$ by the formula:
\begin{equation*}
    \widehat{\mu}(\gamma)=\int_{G}\gamma(-x)d\mu(x).
\end{equation*}
It can be easily verified that the mapping $\mu\mapsto\widehat{\mu}(\gamma)$ for fixed $\gamma\in\widehat{G}$ is a multiplicative linear functional on $M(G)$ which allows to treat the Fourier-Stieltjes transform as the restriction of the Gelfand transform to $\widehat{G}\subset\triangle(M(G))$. Moreover, it is not difficult to show that $\widehat{G}\subset\partial(A)$. On the other hand, $\partial(A)\neq \triangle(M(G))$ (see Corollary 8.2.4. in \cite{gm}) and therefore we are not allowed to use Corollary \ref{corc} and Proposition \ref{propc}. There is no efficient description of all multiplicative-linear functionals of $M(G)$ and the spectral properties of both particular elements of $M(G)$ and the whole Gelfand space of $M(G)$ are very complicated. The most prominent fact is the so-called \textbf{Wiener-Pitt phenomenon} which can be equivalently formulated as $\overline{\widehat{G}}\neq \triangle(M(G))$. We will not elaborate on this topic here and the interested reader is advised to consult Chapters 5-8 in \cite{gm}, for more recent developments, check \cite{owg}. 
\\
Another interesting feature of $M(G)$ is the fact that $M(G)=M(L^{1}(G))$ (the multiplier algebra of $L^{1}(G)$) the latter being defined as the set of all mappings $T:L^{1}(G)\to L^{1}(G)$ satisfying $f\ast(Tg)=(Tf)\ast g$ for every $f,g\in L^{1}(G)$. It is a standard fact in harmonic analysis sometimes called Wendel's theorem. 
\begin{thm}\label{mspm}
    Let $G$ be a locally compact Abelian group and let $\mu\in M(G)$ be a non-invertible measure. Then $\mu$ is a topological divisor of zero.
\end{thm}
We will split the proof into two cases in order to emphasize the difference between compact and non-compact groups.
\\
1. \underline{$G$ is compact}.
\\
If there exists $\gamma_{0}\in\widehat{G}$ such that $\widehat{\mu}(\gamma_{0})=0$, then since $G$ is compact, the group characters are integrable and we have $\mu\ast\gamma_{0}=0$ which shows that $\mu$ is a divisor of zero. On the other hand, if $\widehat{\mu}(\gamma)\neq 0$ for every $\gamma\in\widehat{G}$, then we consider an operator $T_{\mu}:L^{1}(G)\to L^{1}(G)$ given by the formula
\begin{equation*}
    T_{\mu}(f)=\mu\ast f\text{ for }f\in L^{1}(G).
\end{equation*}
Note that $T_{\mu}(\gamma)=\mu\ast\gamma=\widehat{\mu}(\gamma)\gamma$ for every $\gamma\in\widehat{G}$. Since we have assumed that $\widehat{\mu}(\gamma)\neq 0$ for every $\gamma\in\widehat{G}$ we clearly see that the image of $T_{\mu}$ contains all trigonometric polynomials (linear combinations of characters). Therefore, $T_{\mu}$ has a dense range. Clearly, $T_{\mu}$ is also injective as if $f\in L^{1}(G)$ satisfied $\mu\ast f=0$, then we would have $\widehat{\mu}\cdot\widehat{f}=0$ giving $\widehat{f}=0$ which implies $f=0$. Now, we refer to the classical Wendel theorem, which gives $M(G)=M(L^{1}(G))$. It follows that 
\begin{equation*}
\sigma_{M(G)}(\mu)=\sigma_{M(L^{1}(G))}(T_{\mu})=\sigma_{B(L^{1}(G))}(T_{\mu}).
\end{equation*}
The first equality follows from Wendel's theorem, and the second is an elementary exercise (one has to prove that the inverse of an invertible multiplier operator is a multiplier operator).
Hence, an operator $T_{\mu}$ is non-invertible. Having dense range and being injective, an operator $T_{\mu}$ cannot be bounded below\footnote{An operator $T:X\to X$ is bounded below if there exists $c>0$ such that the inequality $\|Tx\|\geq c\|x\|$ holds true for every $x\in X$.}. This finally gives
\begin{equation*}
    \inf\{\|\mu\ast f\|_{M(G)}:f\in L^{1}(G),\text{ }\|f\|_{L^{1}(G)}=1\}=0.
\end{equation*}
2. \underline{$G$ is non-compact}.
\\
In this situation, we are not allowed to deduce that if $\widehat{\mu}(\gamma_{0})=0$, then $\mu$ is a divisor of zero by Theorem \ref{sd}. But it follows that $\mu$ is in the kernel of a group character (treated as multiplicative-linear functional on $M(G)$) and since $\widehat{G}\subset\partial(M(G))$ we infer that $\mu$ is a topological divisor of zero. Therefore, in the rest of the proof, we assume that $\widehat{\mu}(\gamma)\neq 0$ for every $\gamma\in\widehat{G}$. Observe that 
\begin{equation*}
    I:=T_{\mu}(L^{1}(G))=\{\mu\ast f:f\in L^{1}(G)\}\text{ is an ideal in $L^{1}(G)$.}
\end{equation*}
Suppose that $I$ is a proper ideal. The natural way to proceed here is to use some standard algebraic fact such as `every ideal is contained in a maximal ideal'. However, $L^{1}(G)$ is not unital and in order to continue we need to refer to a more subtle result which is a variant of Wiener's tauberian theorem (see Section 7.2 in \cite{r1}).
\begin{thm}
    Every proper closed ideal in $L^{1}(G)$ is contained in a maximal modular ideal\footnote{An ideal $I$ is called a maximal modular ideal, if it is maximal and the quotient algebra is unital. In the case of $L^{1}(G)$ the maximal modular ideals are identified with the elements of $\widehat{G}$.}
\end{thm}
It is now enough to consider $\overline{I}$. It is a closed ideal in $L^{1}(G)$. If $\overline{I}$ is proper, then by the aforementioned theorem $\overline{I}$ is contained in a maximal modular ideal. As $\triangle(L^{1}(G))=\widehat{G}$ we get $\widehat{\mu\ast f}(\gamma_{0})=0$ for some $\gamma_{0}\in\widehat{G}$ and every $f\in L^{1}(G)$. This is an obvious contradiction since $\widehat{\mu}(\gamma_{0})\neq 0$ and there exists $f\in L^{1}(G)$ such that $\widehat{f}(\gamma_{0})\neq 0$. Finally, $\overline{I}=L^{1}(G)$ which is equivalent to the density of range of $T_{\mu}$ and we can proceed as in the compact case.
\begin{cor}\label{mm}
    For any locally compact Abelian group $G$ the measure algebra $M(G)$ has a spectral permanence property.
\end{cor}
\begin{rem}
    The proof in the compact case relies on the fact that the residual spectrum of a multiplier operator on $L^{1}(G)$ is empty. This was originally observed by M. Zafran in greater generality (see \cite{zaf}).
\end{rem}
In fact, the assertion of Theorem \ref{mspm} can be slightly improved.
\begin{thm}
    Let $G$ be a locally compact Abelian group and let $\mu\in M(G)$ be a non-invertible measure. Then there exists a sequence $(f_{n})_{n=1}^{\infty}\subset L^{1}(G)$ with $\|f_{n}\|_{L^{1}(G)}=1$ for every $n\in\mathbb{N}$ satisfying:
    \begin{equation*}
        \lim_{n\to\infty}\|\mu\ast f_{n}\|_{M(G)}=0.
    \end{equation*}
\end{thm}
\begin{proof}
    From Theorem \ref{mspm} we obtain the existence of a sequence $(\mu_{n})_{n=1}^{\infty}\subset M(G)$ with $\|\mu_{n}\|_{M(G)}=1$ for every $n\in\mathbb{N}$ satisfying:
    \begin{equation*}
        \lim_{n\to\infty}\|\mu\ast\mu_{n}\|_{M(G)}=0.
    \end{equation*}
    Consider operators $T_{n}:L^{1}(G)\to L^{1}(G)$ given by the formula $T_{n}(f)=\mu_{n}\ast f$. By Wendel's theorem $\|T_{n}\|_{B(L^{1}(G))}=\|\mu_{n}\|_{M(G)}=1$. Hence, for every $n\in\mathbb{N}$ there exists $g_{n}\in L^{1}(G)$ with $\|g_{n}\|_{L^{1}(G)}=1$ such that \begin{equation*}
        \|T_{n}(g_{n})\|_{L^{1}(G)}=\|\mu_{n}\ast g_{n}\|_{L^{1}(G)}>1-\frac{1}{n}.
    \end{equation*}
    Since $L^{1}(G)$ is an ideal in $M(G)$ we are allowed to define $f_{n}:=\frac{\mu_{n}\ast g_{n}}{\|\mu_{n}\ast g_{n}\|}\in L^{1}(G)$. Then of course $\|f_{n}\|_{L^{1}(G)}=1$ and
    \begin{gather*}
        \|\mu\ast f_{n}\|_{M(G)}=\|\mu\ast\mu_{n}\ast g_{n}\|_{L^{1}(G)}\cdot (\|\mu_{n}\ast g_{n}\|)^{-1}< \\
        <\|\mu\ast\mu_{n}\|_{M(G)}\cdot\|g_{n}\|_{L^{1}(G)}\cdot\frac{n}{n-1}<2\|\mu\ast\mu_{n}\|_{M(G)}.
    \end{gather*}
\end{proof}
\section{Concluding remarks and open problems}
We start with the most general open problem.
\begin{prob}\label{1}
Is the spectral permanence property of a unital commutative Banach algebra $A$ equivalent to $A$ being a cortex algebra? 
\end{prob}
As this problem seems to be the most influential some additional comments are in place. Recall that a subset $S$ of a unital commutative Banach algebra $A$ consists of \textbf{joint topological zero divisors} if $d(x_{1},\ldots,x_{n})=0$ for any finitely many $x_{1},\ldots,x_{n}\in S$, where
\begin{equation*}
    d(x_{1},\ldots,x_{n})=\inf\left\{\sum_{j=1}^{n}\|x_{j}y\|:y\in A,\text{ }\|y\|=1\right\}.
\end{equation*}
The main difference between the set of topological divisors of zero and the set consisting of joint topological divisors of zero lays in the algebraic properties. Clearly, the sum of two topological divisors of zero may be invertible. It is not possible when we analyze joint topological divisors of zero as the following lemma shows (check Lemma 3.4.12 in \cite{k}).
\begin{lem}
    Let $M$ be a subset of a commutative Banach algebra $A$ consisting of joint topological divisors of zero. Then the closed ideal generated by $M$ also consists of joint topological divisors of zero.
\end{lem}
The characterization of multiplicative-linear functionals belonging to the cortex has been given by V. M\"uller in \cite{m} summarizing the research of many authors.
\begin{thm}
    Let $A$ be a commutative unital Banach algebra. Then $\varphi\in\triangle(A)$ belongs to $\mathrm{cor}(A)$ if and only if $\mathrm{ker}\varphi$ consists of joint topological divisors of zero.
\end{thm}
The second problem concerns measure algebra.
\begin{prob}\label{2}
    Let $G$ be a locally compact Abelian group. Is $M(G)$ a cortex algebra?
\end{prob}
In view of Corollary \ref{mm} the positive solution to Problem \ref{1} implies the positive solution to Problem \ref{2} but not the other way round.
\\
In Section 3 we have been exploiting the equality $M(G)=M(L^{1}(G))$. It is natural to consider $M(L^{p}(G))$ for the compact Abelian groups $G$ and $p\in(1,\infty)$. In a similar manner as in Section 3 we can prove the following proposition.
\begin{prop}
    Let $G$ be a compact Abelian group and let $T\in M(L^{p}(G))$ be a non-invertible multiplier for some $p\in (1,\infty)$. Then there exists a sequence $(f_{n})_{n=1}^{\infty}\subset L^{p}(G)$ with $\|f_{n}\|_{L^{p}(G)}=1$ such that
    \begin{equation*}
        \lim_{n\to\infty}\|T(f_{n})\|_{L^{p}(G)}=0.
    \end{equation*}
\end{prop}
The obstacle in deducing that $M(L^{p}(G))$ has a spectral permanence property is the fact that the convolution operator $M_{f}:L^{p}(G)\to L^{p}(G)$ associated with $f\in L^{p}(G)$ and given by the formula $M_{f}(g)=f\ast g$ for $g\in L^{p}(G)$ may have a significantly smaller norm than $\|f\|_{L^{p}(G)}$. It can be easily seen as $\|f\ast g\|_{L^{p}(G)}\leq \|f\|_{L^{1}(G)}\cdot \|g\|_{L^{p}(G)}$. This leads to the following open problem.
\begin{prob}
    Let $G$ be a compact Abelian group. Is $M(L^{p}(G))$ for $p\in(1,\infty)$ a cortex algebra? Does it have a spectral permanence property?
\end{prob}
The last problem is connected to the Fourier-Stieltjes algebras. Let $G$ be a locally compact group. By $B(G)$ and $A(G)$ we denote the Fourier-Stieltjes algebra and the Fourier algebra of $G$ (respectively). For details on these two objects, we refer the reader to \cite{e} and \cite{kl}. We only remark that if $G$ is abelian, then $B(G)$ is canonically isomorphic to $M(\widehat{G})$ and therefore $B(G)$ can be thought of as a generalization of the measure algebra for non-Abelian groups. Moreover, if $G$ is amenable, then $B(G)=M(A(G))$ (see \cite{kl}) and since the variant of the Wiener's tauberian theorem is also true for $A(G)$ (this was already proved in \cite{e}) we are able to repeat the argument from Section 3 to obtain the following result.
\begin{thm}
    Let $G$ be a locally compact amenable group. Then $B(G)$ has a spectral permanence property.
\end{thm}
Since $B(G)=M(A(G))$ holds true iff $G$ is amenable, it would be natural to formulate a hypothesis on a strong connection between the amenability of $G$ and the spectral permanence property of $B(G)$. However, an example of a non-amenable group $G=SL(2,\mathbb{Z})$ with $\triangle(B(G))=G\cup\{\infty\}$ shows that there is no clear relation between this two notions.
\begin{prob}
    Let $G$ be a locally compact group. Which properties of $G$ are responsible for the spectral permanence property of $B(G)$? When is $B(G)$ a cortex algebra?
\end{prob}

\end{document}